\documentclass[11pt]{amsart}

\usepackage[english]{babel}
\usepackage{amsmath,amscd}
\usepackage{amsfonts}
\usepackage{amssymb}
\usepackage{tikz-cd} 
\usepackage{enumerate}
\usepackage{color}
\usepackage{todonotes}

\newtheorem{teo}{Theorem}[section]
\newtheorem{lemma}{Lemma}[section]

\newtheorem{rem}{Remark}[section]
\newtheorem{prop}{Proposition}[section]
\newtheorem{cor}{Corollary}[section]

\newtheorem{defin}{Definition}[section]

\newcommand{\Z}{{\mathbb{Z}}}
\newcommand{\C}{{\mathbb{C}}}
\newcommand{\R}{{\mathbb{R}}}

\newcommand{\pp}{{\mathbb{P}}}
\newcommand{\G}{{\mathcal{G}}}

\newcommand{\A}{{\mathcal{A}}}

\newcommand{\emme}{{\mathcal M}}
\newcommand{\elle}{{\mathcal L}}

\title{Cohomology rings of compactifications of toric arrangements}
\author{Corrado De Concini, Giovanni Gaiffi
}
\date{\today}

\begin{document}
\begin{abstract}
Some projective  wonderful models for  the complement of  a toric arrangement in a $n$-dimensional algebraic torus $T$ 
 were constructed in \cite{DCG}. 
In this paper  we describe their integer cohomology rings by generators and relations.
\end{abstract}
\maketitle

\section{Introduction}

Let \(T\) be   a $n$-dimensional algebraic torus $T$ over the complex numbers, and  let $X^*(T)$ be its  character group, which is a lattice of rank \(n\).

A {\em layer} in \(T\) is the subvariety 
$$\mathcal K_{\Gamma,\phi}=\{t\in T|\, \chi(t)=\phi(\chi),\, \forall \chi\in \Gamma\}$$
where $\Gamma$ is a split direct summand of  $X^*(T)$ and  $\phi:\Gamma \to \mathbb C^*$ is a homomorphism.

A  toric arrangement \(\A\) is given by   finite set of  layers \(\A=\{\mathcal K_{1},...,\mathcal K_{m}\}\) in $T$; if for every \(i=1,...,m\) the layer \(\mathcal K_i\)  has codimension 1
the arrangement  \(\A\) is called  {\em divisorial}.

In \cite{DCG}  it is shown   how to construct  {\em projective wonderful models} for  the complement \(\emme(\A)=T-\bigcup_i \mathcal K_i\).  A projective wonderful model    is a smooth projective  variety \ containing \(\emme(\A)\) as an open set and such that the complement of  \(\emme(\A)\) is  a divisor with normal crossings and smooth irreducible components. We recall that the problem of finding a  wonderful model for \(\emme(\A)\) was first studied by Moci in \cite{mociwonderful}, where a construction of  non projective  models was described.

In this paper we compute  the integer cohomology ring  of the projective wonderful models  by giving an explicit description of their  generators and relations. 
This allows for an extension to the setting of toric arrangements of a  rich theory that regards models of subspace arrangements and was originated in  \cite{DCP2}, \cite{DCP1}. 
In these papers    De Concini and Procesi constructed   wonderful  models  for the   complement of a subspace arrangement, providing both a projective and a non projective version of their construction.  In \cite{DCP1} they showed, using a description of the cohomology rings of the projective  wonderful models  to give an explicit presentation of  a Morgan algebra, that the mixed Hodge numbers and the rational homotopy type of the complement of a complex  subspace arrangement  depend only on the intersection lattice (viewed as a ranked poset). 
The cohomology rings   of the models of subspace arrangements were then   studied  in  \cite{yuzBasi}, \cite{GaiffiBlowups}, were some integer bases were provided,  and also, in the  real case,  in \cite{etihenkamrai}, \cite{rains}.   Some  combinatorial objects (nested sets, building sets) turned out to be relevant in the description of the boundary of the models and of their cohomology rings: their relation with discrete geometry were pointed out   in \cite{feichtner},  \cite{gaiffipermutonestoedra};
 the case of complex reflection groups was dealt with  in \cite{hendersonwreath} from the representation theoretic point of view and in \cite{callegarogaiffilochak} from the homotopical point of view. 
 
 The connections between   the geometry of these models and   the Chow rings of matroids were  pointed out first in \cite{feichtneryuz} and  then in  \cite{adiprasitokatzhuh}, where they  also played a crucial role in the study of   some  log-concavity problems.

As it happens for the case of subspace arrangements, in addition to the interest in  their  own  geometry,  the projective wonderful models of a toric arrangement \(\A\)  may also  spread a new light on the geometric properties of the complement \(\emme(\A)\). 
For instance, in the divisorial  case, using the properties of a projective wonderful model,  Denham and Suciu showed in \cite{DS} that   \(\emme(\A)\) is both a duality space and an abelian duality space.

Let us now describe more in detail the content of the present  paper.

In Section \ref{sec:wonderfulmodels} we  briefly recall   the construction of wonderful models of varieties   equipped with an  {\em arrangement of subvarieties}: this is a generalization, studied  by  Li  in \cite{li},  of the  De Concini and Procesi's construction for subspace arrangements. Its relevance in our setting is explained by the following remark.  In \cite{DCG}, as a first step, the torus \(T\) is embedded in a  smooth projective toric variety \(X\). This toric variety, as we recall in Section \ref{pipo},  is chosen in such a way that the set made  by  the connected components of the intersections of the closures of the layers of  \(\A\) turns out to be  an arrangement of subvarieties \(\mathcal L'\)  and one can apply Li's construction in order to get a projective wonderful model.

More precisely, there are many possible projective wonderful models associated to \(\mathcal L'\), depending on the choice of a {\em building set} for \(\mathcal L'\).

We devote  
 Section \ref{generalpropertiesbuilding} to a  recall of the definition and the main properties of  building sets and  nested sets of arrangements of subvarieties.  These combinatorial objects were introduced by De Concini and Procesi in \cite{DCP2} and  their properties in the case of arrangements of subvarieties were investigated in \cite{li}.  If \(\G\) is a building set for  \(\mathcal L'\) we will denote by \(Y(X,\G)\) the  wonderful model constructed starting from \(\G\).
  
In Section \ref{wellconnectedbuilding}, given any arrangement of subvarieties in a variety \(X\),  we focus on   its  {\em well connected building sets}: these are building sets  that satisfy an additional property,  that will be crucial for our cohomological computations.

In Section \ref{Chern1}  
 we recall a key lemma, due to Keel, that allows to compute  the cohomology ring of the  blowup of a variety \(M\) along a center \(Z\) provided that the restriction map \(H^*(M)\rightarrow H^*(Z)\) is surjective.  In this result the Chern polynomial of the normal bundle of \(Z\) in \(M\) plays a crucial role.  Then we  go back to the case of toric arrangements and, given a smooth projective toric variety \(X\) associated to the toric arrangement \(\A\), we  describe the properties of some polynomials in \(H^*(X,\Z)\) that are related to the Chern polynomials of the closures of the layers of \(\A\) in \(X\).

In  Section \ref{seccohomology} we prove our  main result (Theorem \ref{teopresentazionecoomologia}):  we provide a presentation of the cohomology ring  \(H^*(Y(X,\G),\Z)\) by generators and relations, as a quotient of a polynomial ring over \(H^*(X,\Z)\), whose  presentation is well known.  
A concrete choice for the generators that appear in our theorem is provided in Section \ref{subsecsceltapolinomi}.
We recall  that a description   of the cohomology of a wonderful model of subvarieties as a module  was already found by Li in    \cite{li2}.

Finally, in Section \ref{seccohomologystrata} we  provide a presentation of the cohomology rings of all the strata  in the boundary of  \(Y(X,\G)\).

\section{Wonderful models of stratified varieties}
\label{sec:wonderfulmodels}

In this section we are going to recall  the definitions of arrangements of subvarieties, building sets and nested sets   given in Li's paper  \cite{li}.
We will give these definitions in two steps, first for simple arrangements of subvarieties, then in a more general situation. 
We are going to work over the complex numbers, hence all the algebraic varieties we are going to consider are complex algebraic varieties.
\begin{defin}
\label{def:simple} Let $X$ be a non singular variety. A {\em simple  arrangement of subvarieties} of \(X\) is a finite set \(\Lambda = \{\Lambda_i\}\) of nonsingular closed connected subvarieties \(\Lambda_i\), properly contained in \(X\), which satisfy the following conditions:\\
(i)  \(\Lambda_i\) and \(\Lambda_j\)  intersect  {\em cleanly},  i.e. their intersection  is nonsingular and for every \(y\in  \Lambda_i\cap \Lambda_j\) we have 
\[T_{\Lambda_i\cap \Lambda_j,y} = T_{\Lambda_i,y}\cap T_{\Lambda_j,y}\]
(ii) \(\Lambda_i \cap \Lambda_j\)  either   belongs to  \(\Lambda\)  or is empty.

\end{defin}

\begin{defin}
Let \(\Lambda\) be a simple  arrangement of subvarieties of  \(X\).  A subset \(\G \subseteq \Lambda\) is called a {\em building set} for \(\Lambda\)  if for every \(\Lambda_i \in \Lambda - \G\)  the minimal elements in \(\{G  \in \G \: : \: G \supseteq \Lambda_i\}\)  intersect transversally and their intersection is \(\Lambda_i\). These minimal elements are called the {\em \(\G\)-factors} of \(\Lambda_i\).

\end{defin}

\begin{defin}
Let \(\G\) be  a building set for a simple arrangement \(\Lambda\).
A non empty subset \({\mathcal T}\subseteq \G\) is called \(\G\)-nested if for any subset  \(\{A_1, ...,A_k \}\subset {\mathcal T}\) (with \(k>1\)) of pairwise non comparable  elements, $A_1,\ldots ,A_k$ are the \(\G\)-factors of an element in \(\Lambda\). 

\end{defin}

We remark that in Section 5.4 of \cite{li} the following    more general definitions are provided, to include the case when the intersection of two strata is a disjoint union of strata. 

\begin{defin}
\label{def:nonsimple} An {\em arrangement of subvarieties} of a nonsingular variety \(X\) is a finite set \(\Lambda = \{\Lambda_i\}\) of nonsingular closed connected subvarieties \(\Lambda_i\), properly contained in \(X\), that satisfy the following conditions:\\
(i)  \(\Lambda_i\) and \(\Lambda_j\)  intersect  cleanly;\\
(ii) \(\Lambda_i \cap \Lambda_j\)    is either  equal to the disjoint union of some  of the  \(\Lambda_k\)'s  or it is empty.

\end{defin}
Given an open set $U\subset X$, and a family   $\Lambda$ of subvarieties of $X$, by the restriction $\Lambda_{|U}$ of $\Lambda$ to $U$ we shall mean the family of non empty intersections of elements of $\Lambda$ with $U$.
\begin{defin}
\label{defbuildinggenerale}
Let \(\Lambda\) be an arrangement of subvarieties of  \(X\).  A subset \(\G \subseteq \Lambda\) is called a building set for \(\Lambda\)  if there is an open cover \(\{U_i\}\) of \(X\) such that:\\
a) the restriction of the arrangement \(\Lambda\) to \(U_i\) is simple for every \(i\);\\
b) \( \G_{|U_i}\)  is a building set for \(\Lambda_{|U_i}\).

\end{defin}
We have first introduced the notion of arrangement of subvarieties and then defined  a building set for the arrangement. However it is often convenient to go in the opposite direction and first introduce the notion of building set and use it to define the corresponding arrangement.

\begin{defin}
A finite set \(\G\) of connected subvarieties of \(X\) is called a building set if the set of the connected components of all the possible intersections of collections of subvarieties from \(\G\) is an arrangement of subvarieties \(\Lambda\) (the arrangement {\em induced} by \(\G\))  and \(\G\) is a building set for \(\Lambda\).
\end{defin}
Let us now introduce the notion of $\mathcal G$-nested set in the  more general context of (not necessarily simple) arrangements of subvarieties.
\begin{defin}
Let \(\G\) be  a building set for an arrangement \(\Lambda\).
A subset \({\mathcal T}\subseteq \G\) is called \(\G\)-nested if there is an open cover \(\{U_i\}\) of \(X\) such that, for every \(i\),  \(\G_{|U_i}\)  is simple and 
\({\mathcal T}_{|U_i}\) is  \(\G_{|U_i}\)-nested.
\end{defin}
\begin{rem}
\label{reminsiemevuoto}
We notice that, according to the definition above,  if some  varieties   \(G_1,G_2,..,G_k\in \G\) have empty intersection, then they cannot belong to the same \(\G\)-nested set.
\end{rem}

Once we  have an arrangement  \(\Lambda\) of a nonsingular variety \(X\) and a  building set \(\G\) for \(\Lambda\), we can construct a wonderful model \(Y(X,\G)\)   by  considering  (by  analogy with \cite{DCP1})  the closure of the image of the  locally closed embedding
\[\left( X-\bigcup_{\Lambda_i\in \Lambda}\Lambda_i \right ) \rightarrow \prod_{G\in \G}Bl_GX\] where \(Bl_GX\) is the blowup of \(X\) along \(G\).

In \cite{li}, Proposition 2.8, one shows:
\begin{prop} Let    \(\G\) be a building set in the variety $X$. Let $F\in \G$ be a minimal element in $\G$ under inclusion. Then the set $\G'$ consisting of the proper transforms of the elements in $\G$ is a building set in $Bl_FX$.
\end{prop}
\begin{proof} In fact Li shows this for a building set of a simple arrangement. But since the  definition of building set is local, one can easily adapt his proof  (see also Section 5.4  of \cite{li}).\end{proof}
Using this in  \cite{li}, Theorem 1.3  and the discussion following it, one shows

\begin{teo}
\label{teo:listabuilding}
Let \(\G\) be a 
 building set of subvarieties in a nonsingular variety \(X\). Let us arrange the elements \(G_1,G_2,...,G_m\)  of \(\G\) in such a way that for every \(1\leq i \leq N\) the set \(\mathcal G_h=\{G_1,G_2,\ldots , G_h\}\) is building.  Then if for each $0\leq h\leq m$, we set $X_0=X$ and  \(X_h:=Y(X,\G_h)\), for $h>0$, we have 
\[X_h=Bl_{\tilde G_h}X_{h-1},\]
where   \({\tilde G_h}\) denotes the dominant transform\footnote{In the blowup of a variety \(M\) along a center \(F\)  the dominant transform of a subvariety \(Z\) coincides with the strict transform if \(Z\not\subset  F\) (and therefore it is  isomorphic to the blowup of \(Z\) along \(Z\cap F\)) and  to \(\pi^{-1}(Z)\) if \(Z\subset F\), where \(\pi\: : \: BL_FM\rightarrow M\) is the projection.} of $G_h$ in $X_{h-1}$.
\end{teo}

\begin{rem}\label{pippopippis}
1. We notice that any total ordering  of the elements of a building set \(\G=\{G_1,\ldots ,G_m\}\) which refines the ordering by inclusion, that is $i<j$ if  $G_i\subset G_j$, 
 satisfies  the  condition of Theorem \ref{teo:listabuilding}. 
 
 2. In particular  using  the above ordering  we deduce that $Y(X,\G)$ is obtained from $X$ by a sequence of blow ups each with center a minimal element   in a suitable building set. For every element  \(G\in \G\) we denote by \(D_G\) its dominant transform, that is a divisor  of    $Y(X,\G)$.
\end{rem}

To finish this section let us mention a further result of  Li describing the boundary of \(Y(X,\G)\) in terms of \(\G\)-nested sets:
\begin{teo}[see \cite{li}, Theorem 1.2] 
\label{teoremabordoLi}

The complement in \(Y(X,\G)\) to \(X-\bigcup_{\Lambda_i\in \Lambda}\Lambda_i \) is the union of the divisors \(D_G\), where \(G\) ranges among the elements of  \( \G\).  An intersection of these divisors   
is nonempty if and only if \(\{T_1,...,T_k\}\) is \(\G\)-nested. If the intersection  is nonempty it is transversal.

\end{teo}

\section{Some further properties of building sets}
\label{generalpropertiesbuilding}

In this section we collect a few facts of a technical nature which will be used later. Let  $\Lambda$ be an arrangement of subvarieties in a connected  nonsingular variety \(X\). Let \(\G\) be a  building set for  $\Lambda$  and let \(F\) be a minimal element in \(\G\). Let us denote by \(\widetilde X\) the blowup \(Bl_FX\) and, for every 
subvariety \(D\),  let us  call  \(\widetilde D\)  the transform of \(D\).

Let us  first recall the following lemma from \cite{li} (originally stated for $\Lambda$ simple arrangement,  but  valid also for the general case due to its local nature).
\begin{lemma}[see \cite{li} Lemma 2.9]
\label{lemmali}
Let \(\G\) be a  building set for $\Lambda$, and let \(F\) be a minimal element in \(\G\). Let consider the blowup  \(\widetilde X=Bl_FX\), and let \(A,B,A_1,A_2,B_1,B_2\) be nonsingular subvarieties of \(X\).

\begin{enumerate}[1.]
\item Suppose that \(A_1\not\subset A_2\) and \(A_2\not\subset A_1\), and suppose that \(A_1\cap A_2=F\) and the intersection is clean. Then \(\widetilde A_1\cap \widetilde A_2=\emptyset\).
\item Suppose that \(A_1\) and \(A_2\) intersect cleanly and that \(F\subsetneq A_1\cap A_2\). Then  \(\widetilde A_1\cap \widetilde A_2=\widetilde{A_1\cap A_2}\).
\item Suppose that \(B_1\) and \(B_2\) intersect cleanly and that \(F\) is transversal to \(B_1,B_2\) and \(B_1\cap B_2\).
 Then  \(\widetilde B_1\cap \widetilde B_2=\widetilde{B_1\cap B_2}\).
\item Suppose that \(A\) is transversal to \(B\), \(F\) is transversal to \(B\) and \(F\subset A\). Then  \(\widetilde A\cap \widetilde B=\widetilde{A\cap B}\).

\end{enumerate}

\end{lemma}
The following simple lemma will be useful later.
\begin{lemma}
\label{lemmacontrolloaperti}
Let \(\G\) be a  building set for $\Lambda$, and let  \(U\) be an open set as in the Definition \ref{defbuildinggenerale}. Let us consider two  subsets \(\{H_1,...,H_k\}\) and \(\{G_1,...,G_s\}\) of  \(\G\). If \(H^0=U\cap \bigcap_{i=1,...,k}H_i\neq \emptyset\) and  \[H^0=U\cap \bigcap_{i=1,...,k}H_i\subset G^0=U\cap \bigcap_{j=1,...,s}G_j\]
then the connected component of \(\bigcap_{i=1,...,k}H_i\) that contains \(H^0\) is contained in the connected component of \(\bigcap_{j=1,...,s}G_j\) that contains \(G^0\).
\end{lemma}
\begin{proof}
First we notice that  \(H^0\) and \(G^0\) are connected by the Definition \ref{defbuildinggenerale}. The  statement follows since \(H^0\) is a dense open set of the connected component of \(\bigcap_{i=1,...,k}H_i\) that contains it.

\end{proof}

\begin{prop}
\label{lemmaintersezione} Let \(\G\) be a building set for  $\Lambda$.  Let us fix an open set \(U\) as in the Definition \ref{defbuildinggenerale} (for brevity, in what follows every object will be restricted to \(U\) but we are going to omit the symbol of restriction, for instance we will denote by \(G\) the set \(G\cap U\) for every \(G\in \G\)). 
 Let \(G_1,G_2\in  \G\) be not comparable. Then either  \(G_1\cap G_2=\emptyset\),  or \(G_1\cap G_2\in \G\) or \(G_1\cap G_2\) is transversal.
\end{prop}
\begin{proof}

Let us suppose \(G_1\cap G_2\neq \emptyset\). We know by the definition of building set that 
\begin{equation}
\label{eqintersezioni} G_1\cap G_2=H_1\cap H_2\cap ...\cap H_k \end{equation}
where the \(H_j\)'s are the minimal elements in \(\G\) that contain \(G_1\cap G_2\) and the intersection among the \(H_j\)'s is transversal.
We can suppose, up to reordering,  that \(H_1\subset G_1\). 

If we also have \(H_1\subset G_2\) then \(H_1\subset G_1\cap G_2\), while from the equality (\ref{eqintersezioni})  we have \(G_1\cap G_2\subset H_1\). This means that \(G_1\cap G_2=H_1\) and therefore it belongs to \(\G\).

 If, on the other hand, \(H_1\) is not contained in \( G_2\),   we can suppose, up to reordering,  that \(H_2\subset G_2\).
Then \(H_1\cap H_2\subset G_1\cap G_2\) while from the equality (\ref{eqintersezioni}) we have \(G_1\cap G_2\subset H_1\cap H_2\). This means that \(G_1\cap G_2= H_1\cap H_2\) so that in particular \(k=2\).

Since the intersection \(H_1\cap H_2\) is transversal, then also \(G_1\cap G_2\) is transversal. Indeed once one fixes a point \(y\in H_1\cap H_2 \), the set of linear equations that describe the tangent space \(T_{H_i,y}\)  includes the set of equations that describe    \(T_{G_i,y}\). Since the intersections are clean and all the involved varieties are smooth this implies in particular that  \(G_1=H_1\) and \(G_2=H_2\).

\end{proof}
\begin{cor}[see Lemma 2.6 in \cite{li}]
\label{cortransversale}
Let \(\G\) be a  building set.  Let \(F\) be a  element in \(\G\).
\begin{enumerate}[1.]
\item  If $F$ is minimal, for  any  \(G\in \G \),  either \(G\) contains \(F\), or \(F\cap G=\emptyset\),  or \(F\cap G\) is transversal.
\item 
Let $K$ be an element of the arrangement induced by \(\G\) such that   none of  its    \(\G\) factors contains $F$.
Assume that $H=K\cap F$ also has  $F$ as one of its $\G$ factors. 
Then the intersection of  \(K\) and \( F\) is transversal.
\end{enumerate}
\end{cor}
\begin{proof}

First we notice that, by Lemma \ref{lemmacontrolloaperti}, for every open set \(U\) as in the Definition \ref{defbuildinggenerale},  \(F\cap U\) is empty or it is minimal also for the restriction of \(\G\) to \(U\). Therefore it is sufficient to prove our statement locally (and from now on we will think of every object as intersected with \(U\)).

So  \((1)\) is an immediate consequence of Proposition \ref{lemmaintersezione} since if \(F\not\subset G\)  and \(F\cap G\neq \emptyset\), then \(F\cap G\notin \G\) by minimality of \(F\).

As for   \((2)\),  since \(\G\) is building,    we can write
\[H=B_1\cap ..\cap B_j\cap F\]
where  \(B_1,...,B_j,F\) (with \(j\geq 1\)) are the  \(\G\)  factors of $H$ and their intersection is transversal.

Let $G$ be a $\G$ factor of $K$. Since $G$ contains $H$ but does not contain $F$, it must contain one of the $B_i$'s. It follow that 
$S=B_1\cap\cdots \cap B_j\subset K$.
We deduce that,  since 
$$H=K\cap F=S\cap F,$$ $K$ and $F$ intersect  cleanly and $S$ and $F$ intersect transversally, also $K$ and $F$ intersect transversally.

\end{proof}
\section{Well connected building sets}
\label{wellconnectedbuilding}
In the computation of the cohomology of compact wonderful models  we will need some building sets that have an extra property. 

\begin{defin}
A building set \(\G\)   is called {\em well connected} if  
%
%
for any subset \(\{G_1,...,G_k\}\) in \(\G\),  the intersection \(G_1\cap G_2\cap ...\cap G_k\) is either empty,  or connected  or it is the union of connected components each belonging to \(\G\).

\end{defin}

\begin{rem}
\label{remconnected}
In particular, if \(\G\) is well connected and \(F\in \G\) is minimal, we have that for every \(G\in \G\)  the intersection \(G\cap F\) is either empty or connected.

\end{rem}

Notice that, for example, if $\Lambda$ is an arrangement of subvarieties  then  $\Lambda$ itself  is a, rather obvious, example of a well connected building set.

As another example, if  $\Lambda$ is simple then clearly every building set for $\Lambda$ is well connected.

The following two propositions are going to be crucial in our inductive procedure.
\label{seclemma}
Let \(X\) be a smooth variety
and   \(\G=\{G_1,...,G_m\}\)   a well connected building set of subvarieties of $X$ whose  elements are ordered in a way that refines inclusion.

\begin{prop}
For every \(k=1,...,m\), the set \(\G_k=\{G_1,...,G_k\}\) is a well connected building set.
\end{prop}
\begin{proof}
Let us prove that \(\G_k\) is building. 

First we check what happens `locally'.  We fix an open set    \(U\)  as in the Definition \ref{defbuildinggenerale} and   in what follows we will consider the restriction of every object to \(U\).

Since \(\G\) is building, we know that every intersection \(G_{j_1}\cap\cdots \cap G_{j_s}\) of elements of \(\G_k\)  is equal to the transversal intersection  of the  minimal elements \(B_1,...,B_h\) of \(\G\) that contain \(G_{j_1}\cap\cdots \cap G_{j_s}\). Up to reordering we can assume that the set $\{B_1,\ldots , B_r\}$ for some $r\leq s$ consists of those among the $B_i'$s which are contained in at least one among the $G_{j_t}$'s . Notice that necessarily 
$$\bigcap_{i=1}^rB_j=\bigcap_{i=1}^hB_j=G_{j_1}\cap\cdots \cap G_{j_s}.$$
Since the intersection of $B_1,\ldots ,B_h$ is transversal,  we clearly have that $r=h$ and so 
we deduce that for each $j\leq h$, there is an $a\leq k$ with \(B_j=G_a\).

Going back from the local to the global setting, we observe that with the argument above we have proven that \(B_j\cap U=G_a\cap U\). Since intersecting with $U$ preserves inclusion relations by Lemma \ref{lemmacontrolloaperti},   we immediately deduce that   \(B_i\in \mathcal G_k\)  for each $i=1,\ldots ,h$. 

A similar  reasoning  also shows that \(\G_k\) is  well connected.

\end{proof}


Let us consider the variety \(Z:=G_m\).  Let us take the family $\mathcal H=\{H_1,\ldots ,H_u\}$ of non empty subvarieties in $Z$ which are obtained as connected components of intersections $G_i\cap Z$ with $i<m$. 

Let us remark that, since \(\G\) is well connected, if $G_i\cap Z$ is non connected (and of course non empty)
 its  connected components  belong to \(\G\) so that  each of them equals some $G_j\subsetneq Z$. We deduce that  we do not need to add the connected components of the disconnected intersections   $G_i\cap Z$. In particular $u\leq m-1$. 
 
 We order $\mathcal H$ in such a way that if for each  $1\leq i\leq u$, we set $s_i\leq m-1$ equal to the minimum index such that $H_i= G_{s_i}\cap Z$, $s_i<s_j$ as soon as $i<j$. 

\begin{prop}
\label{lemmarrangiamentoindotto}
The family of subvarieties \(\mathcal H=\{H_1,...,H_u\}\) in \(Z\) 
is building and well connected. \end{prop}
\begin{proof}
Let us prove that \(\mathcal H\) is building. By definition of building set,  it suffices  to prove this locally,  i.e. in \(U\cap Z\) for any of the open sets  \(U\) that appears in the definition of the building set \(\G\). So we  fix such an 
\(U\) and assume that $X=U$.

As before, we write for each $i=1,\ldots ,u$,    $H_i=G_{s_i}\cap Z$ with $G_{s_i}\in \mathcal G_{m-1}$.

Let \(H=H_{i_1}\cap\cdots \cap H_{i_\ell}\) be a nonempty intersection of elements of \(\mathcal H\).
Since \(H\) is also  an intersection of elements of \(G\) we can write
\[H=H_{i_1}\cap\cdots \cap  H_{i_\ell}=  G_{s_{i_1}}\cap\cdots \cap  G_{s_{i_\ell}}\cap Z=G_{j_1}\cap \cdots \cap G_{j_k},\]
where \(G_{j_1}, ..., G_{j_k}\) are the minimal elements of \(\G\) that contain \(H\) and their intersection is transversal in \(X\). 

Consider the set $I=\{s_{i_1},\ldots, s_{i_\ell},m\}$ and $J=\{j_1,\ldots , j_k\}$. In $I\times J$ we take the subset $S$ consisting of those pairs $(a,b)$ such that $G_a\supset G_b$. By eventually reordering the indices, we can assume that the projection of $S$ on the second factor equals $\{j_1,\ldots , j_{k'}\}$, for some $k'\leq k$. On the other hand,
by  minimality,  the projection of $S$ on the first factor is surjective and we can further assume that 
$Z\supset G_{j_{k'}}$.

We claim that $k=k'$. Indeed if $k'$ where less than $k$, 
$$H=H_{i_1}\cap\cdots \cap H_{i_\ell}=G_{j_1}\cap\cdots \cap G_{j_k'}$$
and the intersection $G_{j_1}\cap\cdots \cap G_{j_k}$ would not be transversal.

Let \(\beta_s\), for every  $1\leq s\leq  k$, be  such that \(H_{\beta_s}\in \mathcal H\) is the connected component of \(G_{j_s}\cap Z\) that contains \(H\).

Then we have 
$$ H_{\beta_1}\cap \cdots\cap H_{\beta_{k}}=H.$$
We set $d=k-1$ if $Z= G_{j_{k}}$, $d=k$ otherwise.
In both cases we then easily see that 
$H$ is the transversal intersection 
$ H_{\beta_1}\cap \cdots\cap H_{\beta_{d}}.$

Finally let us   observe that \(H_{\beta_1},H_{\beta_2},...,H_{\beta_{d}}\) are the minimal elements in \(\mathcal H\) containing \(H\). This is obvious if \(d=1\). 
If  \(d>1\),   assume by contradiction that there is  an element  \( H'\in \mathcal H\) and an index \(s\in \{1,...,d\}\) such that \(H\subseteq  H'\subsetneq H_{\beta_s}\).
The last inclusion implies that  
\[H'=G'\cap Z\subsetneq G_{j_s}\cap Z\]
for some $G'\in \mathcal G.$
From this in particular it follows that   \(Z\) is not contained in  \(G'\) and that  \(G_{j_s}\nsubseteq G'\). Now, since the elements \(G_{j_1}, ..., G_{j_k}\) are the minimal elements of \(\G\) that contain \(H\),   \(G_{j_\nu}\subseteq G'\) for some $1\leq \nu\leq k$. Since $Z$ is not contained in $G'$, $j_\nu\neq m$. 

Then we have \(H_{\beta_\nu}\subseteq G_{j_\nu}\cap Z\subseteq G'\cap Z=H'\).
But \(H'\subsetneq H_{\beta_s}\), so we deduce \(H_{\beta_\nu}\subsetneq H_{\beta_s}\) which is a contradiction, since we know that their intersection is transversal.

This completes the proof that \(\mathcal H\) is building. 

Let us now prove that  \(\mathcal H\) is well connected (this  proof  is not local).
First we observe that by definition the elements of \(\mathcal H\) are connected.
Then let \(H=H_{i_1}\cap\cdots \cap H_{i_t}\) be a nonempty intersection of elements of \(\mathcal H\).
Since \(H\)   is also an intersection of elements of \(\G\), by the well connectedness of \(\G\),  if  \(H\)   is not connected then it is the disjoint union of connected components that belong to  \(\G\). Let \(G_s\) be such a component: since it is contained in \(Z\) then \(s<m\) and \(G_s=G_s\cap Z\) belongs to \(\mathcal H\). This proves that all these connected components belong to \(\mathcal H\).

\end{proof}

\begin{rem}
 \label{remarkarrangiamentoindotto} In the proof of the proposition above we have shown that if $H=H_{i_1}\cap\cdots \cap H_{i_\ell}$ then $H$ is equal to the transversal intersection of \(H_{\beta_1},...,H_{\beta_d}\). In particular we have shown that, for every \(\gamma=1,...,d\), \(H_{\beta_\gamma}\) is a connected component of  \( G_{j_\gamma}\cap Z\)
and \(G_{j_\gamma}\) is included in some of the \(G_{s_{i_1}},...,G_{s_{i_\ell}}\). 
 With the chosen ordering of  $\mathcal H=\{H_1,\ldots , H_u\}$,  this implies that each one of the \(\beta_j\)'s is \(\leq max\{i_1,...,i_l\}\).  Therefore we have proven  that for  each  \(1\leq i\leq u\), the arrangement of subvarieties  \(\mathcal H_i=\{H_1,...,H_{i}\}\) in \(Z\) 
is building and well connected.\end{rem}

\begin {prop}\label{zetat} Let $1\leq s\leq m-1$ and let $1\leq i\leq u$ be such that $s_i\leq s<s_{i+1}$. Then the 
proper transform of $Z$ in $X_s$  equals $Z_i.$\end{prop}
\begin{proof} We first treat the case in which $s_i<s<s_{i+1}$. In this case there are two possibilities
\begin{enumerate}\item $G_s\cap Z=\emptyset$.
\item $G_s\cap Z\neq \emptyset$ and each of its connected components lies in $\mathcal G$.

\end{enumerate}
In the first case there is nothing to prove.  In the second case, by assumption when we reach $X_{s-1}$ we have already blown up each of the connected components of $G_s\cap Z$. Since we know that the intersection $G_s\cap Z$ is clean, by Lemma \ref{lemmali}.(1) the  transforms of $Z$ and $G_s$ in $X_{h-1}$ have empty intersection and clearly also in this case there is nothing to prove.

If $s=s_i$ again we have two cases
\begin{enumerate} \item $G_{s_i}\subset Z$. 
\item The intersection $G_{s_i}\cap Z$ is transversal and does not lie in $\mathcal G$.
\end{enumerate}
Let us denote by  $\tilde H_i$ and $\tilde G_{s_i}$   the proper transforms of $  H_i$ and $  G_{s_i}$ in $X_{s_i-1}$

By Remark \ref{pippopippis}.2, $X_{s_i}$ is obtained from    $X_{s_i-1}$ by blowing  $\tilde G_{s_i}$ which is a minimal element in a suitable building set. 

Thus our statement in case (1) follows, using induction from the fact that
$Z_{i}=Bl_{\tilde G_{s_i}}Z_{i-1}=Bl_{\tilde H_i}Z_{i-1}$.

 In  case (2),  since $H_i=G_{s_i}\cap Z$, by induction  and  Lemma \ref{lemmali}  we deduce that $\tilde H_i=\tilde G_{s_i}\cap Z_{i-1}$.
 So by Corollary \ref{cortransversale}.(1), and the minimality of $\tilde G_{s_i}$ in a suitable building set, the intersection $\tilde H_i=\tilde G_{s_i}\cap Z_{i-1}$ is transversal 
 and the proper transform of $Z$ in $X_{s_i}$ equals $Bl_{\tilde H_i}Z_{i-1}=Z_i$ as desired. \end{proof}

\section{Recollections on the construction of  projective wonderful models of a toric arrangement}
\label{pipo}
We are now going to consider a special situation.
We consider  a $n$-dimensional algebraic torus $T$ over the complex numbers and we denote by   $X^*(T)$ its character group. 

Let us  take $V=hom_\Z(X^*(T),\R)=X_*(T)\otimes_\Z\R$,  $X_*(T)$ being the lattice $hom_\Z(X^*(T),\Z)$  of one parameter subgroups in $T$. 

Then, setting $V_{\mathbb C}=hom_\Z(X^*(T),\C)=X_*(T)\otimes_\Z\C$, we have a natural identification of $T$ with $V_\C/X_*(T)$ \and we may consider a $\chi\in X^*(T)$ as a  linear function on $V_\C$. From now on the corresponding character  $e^{2\pi i\chi}$ will be usually denoted by $x_\chi$.

Now, let   \(\A\)  be the toric arrangement  \(\A=\{\mathcal K_{\Gamma_1,\phi_1},...,\mathcal K_{\Gamma_r,\phi_r}\}\) in $T$ as defined in the Introduction, where the \(\Gamma_i\) are split direct summands of  $X^*(T)$ and the \(\phi_i\)'s are homomorphisms \(\phi_i\: : \Gamma_i\rightarrow \C^*\). 

Remark  that $\mathcal K_{\Gamma,\phi}$ is  a coset with respect to the torus 
$H=\cap_{\chi\in \Gamma}Ker(x_\chi)$.
Now we consider the subspace $V_\Gamma=\{v\in V|\, \langle\chi,v\rangle=0,\, \forall \chi\in \Gamma\}$. Notice that since $X^*(H)=X^*(T)/\Gamma$, $V_\Gamma$ is naturally isomorphic to $hom_\Z(X^*(H),\R)= X_*(H)\otimes_{\Z} \R$.

In \cite{DCG} (see Proposition 6.1)  it was shown how to construct a projective smooth   \(T\)-  embedding  \(X=X_\Delta\) whose fan \(\Delta\) in \(V\) has the following property. For every \(\Gamma_i\)   there is an integral  basis of $\Gamma_i$, $\chi_1,\ldots ,\chi_s,$ such that,   for every cone \(C\) of \(\Delta\)  with generators \(r_1,\ldots ,r_h\), up to replace  \(\chi_i\) with \(-\chi_i\) for some \(i\), 
the pairings  
$\langle\chi_i,r_j\rangle$ are all \(\geq 0\)  or all \(\leq 0\).  The basis $\chi_1,\ldots \chi_s$ is called an {\em equal sign basis} for \(\Gamma_i\).

Moreover we remark that  \(\Delta\) can be chosen in such a way that for every layer $\mathcal K_{\Gamma,\phi}$, obtained as a connected component of the intersection of some of the layers in \(\A\), the lattice \(\Gamma\) has an equal sign basis. Given such a \(\Delta\), we will say that \(X=X_\Delta\) is a {\em good toric variety} for  \(\A\).

In such a toric variety  \(X\)  consider the closure  \(\overline {{\mathcal K}}_{\Gamma,\phi}\) of a layer. This  closure turns out to be a toric variety, whose explicit description is provided by the following result from \cite{DCG}. 

\begin{teo}[Proposition 3.1 and Theorem 3.1 in \cite{DCG}]\label{faneH} For every layer    \(\mathcal K_{\Gamma,\phi}\), let \(H\) be the corresponding subtorus and let $V_{\Gamma}=\{v\in V|\, \langle\chi,v\rangle=0,\, \forall \chi\in \Gamma\}$. Then, 
\begin{enumerate}[1.]
\item For every cone $C\in \Delta$, its relative interior is either entirely contained in $V_{\Gamma}$ or disjoint from $V_{\Gamma}$.
\item The collection of cones $C\in \Delta$ which are contained in $V_{\Gamma}$ is a smooth fan $\Delta_{H}$.

\item   \(\overline {{\mathcal K}}_{\Gamma,\phi}\) is a smooth $H$-variety whose fan is $\Delta_{H}$.
\item  Let $\mathcal O$ be a $T$ orbit in $X=X_\Delta$ and let $C_{\mathcal O}\in \Delta$ be the corresponding cone. Then 
\begin{enumerate}
\item If $C_{\mathcal O}$ is not contained in $ V_{\Gamma}$, $\overline {\mathcal O}\cap \overline {{\mathcal K}}_{\Gamma,\phi}=\emptyset$.
\item If $C_{\mathcal O}\subset V_{\Gamma}$, $ {\mathcal O}\cap \overline {{\mathcal K}}_{\Gamma,\phi}$ is the $H$ orbit in $\overline {{\mathcal K}}_{\Gamma,\phi}$ corresponding to 
$C_{\mathcal O}\in \Delta_{H}$. 
\end{enumerate} \end{enumerate}\end{teo}

Let us   denote  by \({\mathcal Q}'\) (resp.  \({\mathcal Q}\) )  the set whose elements are  the subvarieties  \(\overline {{\mathcal K}}_{\Gamma_i,\phi_i}\)   of \(X\) (resp. the subvarieties  \(\overline {{\mathcal K}}_{\Gamma_i,\phi_i}\)  and the irreducible components of the complement \(X- T\)). We then denote by \(\elle'\) (resp. \(\elle\)) the poset  made by all the connected components of all the intersections of some of the  elements of \({\mathcal Q}'\) (resp.  \({\mathcal Q}\) ).
In \cite{DCG} (Theorem 7.1) we have shown that the  family   \(\elle\) is an arrangement  of subvarieties in \(X\). As a consequence also \(\elle'\), being contained in $\elle$ and closed under intersection, is an arrangement  of subvarieties.

We notice that the complement in \(X\) of the union of the elements in \(\elle\) is equal to \(\emme(\A)\), and it is strictly contained in the complement of the union of the elements in \(\elle'\).

In the sequel of this paper we will focus on the wonderful model \(Y(X,\G)\) obtained by choosing a (well connected) building set \(\G\) for \(\elle '\).  Let us now explain  our choice.

As a consequence of  Theorem \ref{faneH} we deduce that the elements of $\mathcal L$ are exactly the non empty intersections 
$\overline{\mathcal K}_{\Gamma, \phi}\cap \overline {\mathcal O}\neq \emptyset.$ This means that they are indexed by a family of  triples $(\Gamma,\phi, C_{\mathcal O})$ with $\phi\in \hom(\Gamma,\mathbb C^*)$, and $C_{\mathcal O}\subset V_\Gamma$.  The triples $(\{0\},0, C_{\mathcal O})$ index the closures of $T$ orbits in $X$.

The intersection $$\overline{\mathcal K}_{\Gamma\phi}\cap \overline {\mathcal O}$$ is transversal. Furthermore, since $X$ is smooth, if the cone $C_{\mathcal O}=C(r_{i_1},\ldots r_{i_h})$, where the $r_{i_j}$ are the rays of $C_{\mathcal O}$, we have that $\overline {\mathcal O}$ is the transversal intersection of the divisors $D_{r_{i_{j}}}$. We deduce:

\begin{prop}\label{ilg+}
Let \(\G\) be a  building set for the arrangement of subvarieties \(\mathcal L'\) in \(X\).
Then  \(\G^+=\G \cup \{D_r\}_{r\in \mathcal R}\) is a  building set for \(\mathcal L\).
\end{prop}
\begin{proof}

We have seen that an element of $S\in \mathcal L$ is the transversal intersection
$$S=\overline{\mathcal K}_{\Gamma,\phi}\cap \bigcap_{r\in J}D_{r},$$
with $J$ a, possibly empty, subset of $\mathcal R$

We know that, in a suitable open set $U$, $\overline{\mathcal K}_{\Gamma,\phi}$ is the transversal intersection of the minimal elements in $\mathcal G$ containing it. Since  \(\G^+=\G \cup \{D_r\}_{r\in \mathcal R}\), the same holds for $S$ with respect to \(\G^+\).

On the other hand  we observe that the connected components of any  intersection of elements of \(\G^+\) belong to  \(\mathcal L\), by the definition of \(\mathcal L\).

This clearly means that \(\mathcal L\) is the arrangement induced by  \(\G^+\) and that \(\G^+\) is a building set {\em for} \(\mathcal L\).
\end{proof}
As a consequence of the proposition above, we can construct   \(Y(X,\G^+)\), which  is a projective wonderful model for the complement \[\emme(\A)=X-\bigcup_{G\in \G^+}G=X-\bigcup_{A\in \elle}A.\]
Now we observe that the varieties \(Y(X,\G)\) and \(Y(X,\G^+)\) are isomorphic.

To prove this for instance one could order     \(\G^+\) in the following way: one puts first the elements of \(\G\) ordered in a way that refines inclusion, then the elements \(D_r\) in any order. As we know from Theorem \ref{teo:listabuilding}, \(Y(X,\G^+)\) can be obtained as the result of a series of blowups starting from \(X\). After the first \(|\G|\) steps we get \(Y(X,\G)\), then the centers of the last \(|\mathcal R|\)  blowups are divisors so \(Y(X,\G^+)\) is isomorphic to \(Y(X,\G)\).

To finish our  recollection  on projective models and toric varieties, we need to   describe explicitly   the restriction map in cohomology
$$j^*:H^*(X,\mathbb Z)\to H^*( \overline {{\mathcal K}}_{\Gamma,\phi},\mathbb Z),$$
induced by the inclusion, for a layer ${\mathcal K}_{\Gamma,\phi}$.

Let us first recall the following well known presentation of the cohomology ring of a smooth projective toric variety by generators and relations. 
Let \(\Sigma\) be a smooth complete fan and let \(X_\Sigma\) its associated toric variety.
Take a one   dimensional face   in $\Sigma$. This face  contains a unique primitive ray $r\in hom_\Z(X^*(T),\Z)$. We denote by $\mathcal R$ the collection of rays. We have:
\begin{prop}(see for example \cite{fultontoric}, Section 5.2.)\label{danilov} $$H^*(X_\Sigma),\mathbb Z)=\mathbb Z[c_{r}]_{r\in \mathcal R}/L_\Sigma$$
where $L_\Sigma$ is the ideal generated by 
\begin{enumerate}
\item[a)] \(c_{r_1}c_{r_2}\cdots c_{r_k}\) if the rays \(r_1,...,r_k\) do not belong to a cone of \(\Sigma\).
\item[b)]  \(\sum_{r\in \mathcal R}\langle \beta, r\rangle  c_r \) for any \(\beta\in X^*(T)\).
\end{enumerate}
Furthermore the residue class of $c_r$ in $H^2(X_\Sigma,\mathbb Z)$ is the cohomology class of the divisor $D_r$ associated to the ray $r$ for each $r\in \mathcal R.$ By abuse of notation we are going to denote by $c_r$ its residue class in $H^2(X_\Sigma,\mathbb Z)$.
\end{prop}
Let us consider as before a toric arrangement \(\A\) in a torus \(T\),  and a good toric variety \(X=X_\Delta\) for \(\A\).
We can apply the proposition above to both $X$ and the closure of a layer $\overline {{\mathcal K}}_{\Gamma,\phi}$.  Let us remark that by  Theorem \ref{faneH}.4, if $r\notin V_{\Gamma}$ then 
the divisor $D_r$ does not intersect $\overline {{\mathcal K}}_{\Gamma,\phi}$, while if $r\in V_{\Gamma}$ the divisor $D_r$   intersects $\overline {{\mathcal K}}_{\Gamma,\phi}$ in the divisor corresponding to $r$.
We deduce:
\begin{prop}\label{sopra} The restriction map 
$$j^*:H^*(X,\mathbb Z)\to H^*( \overline {{\mathcal K}}_{\Gamma,\phi},\mathbb Z).$$
is surjective  and its  kernel  $I$ is generated by the classes $c_r$ with   $r\in \mathcal R$ such that $r\notin V_{\Gamma}$. 
 \end{prop}

\section{A result of Keel and  Chern polynomials of closures of layers}
\label{Chern1}
Let us as before consider a toric arrangement $\mathcal A$ in the torus $T$.
As we recalled in Section \ref{pipo}, we can and will choose  $X=X_\Delta$ to be a good  toric variety associated to $\mathcal A$ and  take  the arrangement $\mathcal L'$ of subvarieties in $X$.

Let us fix  a well connected building set  \(\G=\{G_1,...,G_m\}\)  for  $\mathcal L'$,  ordered in such a way  that   if \(G_i\subsetneq  G_j\)  then \(i<j \).

Our goal is to describe   the cohomology ring \(H^*(Y(X,\G),\Z)\) by generators and relations. For this we are going to use the following result due to Keel.

Let \(Y\) be a smooth variety, and suppose that \(Z\) is a regularly embedded subvariety   of codimension \(d\)  (we denote by \(i\: : \: Z\rightarrow Y\) the inclusion).
Let \(Bl_Z(Y)\) be the blowup of \(Y\) along \(Z\), so we have a map \(\pi\: : \:Bl_Z(Y)\rightarrow Y\),  and let \(E=E_Z\) be the exceptional divisor.

\begin{teo}[Theorem 1 in the Appendix of \cite{keel}]

\label{teoKeel}
Suppose that the map \(i^*\: : \: H^*(Y)\rightarrow H^*(Z)\) is surjective with kernel \(J\), then \(H^*(Bl_ZY)\) is isomorphic to 
\[\frac{H^*(Y)[t]}{(P(t),t\cdot J)  } \] 
where \(P(t)\in H^*(Y)[t]\) is any polynomial whose constant term is \([Z]\) and whose restriction to \(H^*(Z)\) is the Chern polynomial of the normal bundle \(N=N_ZY\), that is to say
\[i^*(P(t))=t^d+t^{d-1}c_1(N)+\cdots+c_d(N)\]
This isomorphism is induced by \(\pi^*\: : \: H^*(Y)\rightarrow H^*(Bl_ZY)\) and by sending \(-t\) to \([E]\).

\end{teo}

In order to use Theorem \ref{teoKeel}, we need to introduce certain polynomials with coefficient in $H^*(X,\mathbb Z)$.

For every  $G:=\overline{\mathcal K}_{\Gamma,\phi}\in \mathcal L'$,
we set  \(\Lambda_G:=\Gamma\).
Setting $B=H^*(X,\mathbb Z)$ we choose a polynomial $P_G(t)=P_G^X(t)\in B[t]$  that satisfies the 
following two properties:
\begin {enumerate}
\item  $P_G(0)$ is the class dual to the class of  $ G$ in homology. 
\item the restriction map to  \(H^*(  {G},\mathbb Z)[t]\) sends $ P_G(t)$ to  the  Chern polynomial of  $N_{ G}X$.
\end{enumerate}
We will say that such a polynomial is a {\em good  lifting} of the Chern polynomial of  $N_{ G}X$.
Let $I$ be the kernel of the restriction map $$j^*:H^*(X,\mathbb Z)\to H^*(  {G},\mathbb Z).$$

\begin{lemma} 
\label{casounavariabile} The ideal  $(tI,  P_G(-t))\subset B[t]$ does not depend on the choice of   $P_G(t)$.
\end{lemma}
\begin{proof} 
Let  $ Q_G(t)$  be another polynomial satisfying  (1) e (2). From  (1) we know that   $P_G(t)- Q_G(t)$ has constant term equal to 0. Moreover from (2) we deduce that  every coefficient of $P_G(t)- Q_G(t)$ belongs to  $I$ so $P_G(t)- Q_G(t)\in (tI)$. 
\end{proof}

Let us now consider two elements \(G,M\in \mathcal L'\), with   $G\subset M$.
Let us choose a polynomial   $\overline P_G^M(t)\in H^*( M,\mathbb Z)[t]$ that is a good lifting of the Chern polynomial of   $N_{ G} M$ (i.e. it satisfies the properties (1) and (2) in $H^*(  M,\mathbb Z)$) and let us denote by $P_G^M(t)$ a lifting of $\overline P_G^M(t)$ to  \(H^*(X,\mathbb Z)[t]\). The existence of such polynomial follows immediately from Proposition \ref{sopra}.

Let us now consider  a well connected building set $\mathcal G=\{G_1,\ldots ,G_m\}$ for the arrangement of subvarieties \(\elle'\) in  \(X\) (see Section \ref{pipo}),  ordered in a way that refines inclusion. 

Now, for every pair  $(i,A)$ with   \(i\in \{1,...,m\}\), and    \(A\subset \{1,...,m\}\) such that if \(j\in A\) then \( G_i\subsetneq G_j\),   we can define  the following polynomial  in  $H^*(X,\mathbb Z)[t_1,\ldots ,t_m]= B[t_1,\ldots ,t_m]$. 

Let us  consider the set  \(B_i=\{h\: | \: G_h\subseteq G_i\}\),
and let us denote by \(M\) the unique connected component of  \(\bigcap_{j\in A}G_j \)  that contains  $G_i$ (if   $A=\emptyset$ we put  $M=X$). Then, after choosing   all the polynomials $P^M_{G_i}$ as explained before, we put:
  
\[ F(i,A)=P^M_{G_i}(\sum_{h\in B_i}-t_{h})\prod_{j\in A}t_{j}. \]

We also include as  special cases the pairs \((0,A)\)  where 
$A$ is such that  \(\bigcap_{j\in A}G_j=\emptyset \), and we define the polynomials:  
\[ F(0,A)=\prod_{j\in A}t_{j}.\]

\begin{prop}
\label{propindependent}
Let  $I_m$ be the ideal in  $B[t_1,\ldots ,t_m]$  generated by  \begin{enumerate}[1.]\item the products 
$t_ic_r$ for every  ray $r\in \Delta $  that does not belong to \(V_{\Lambda_{G_i}}\) (i.e. \(\langle r, \cdot \rangle\) does not vanish on   $\Lambda_{G_i}$); 

\item the polynomials   $F(i,A)$ defined above.
 \end{enumerate}
Then  $I_m$ does not depend on the choice of the polynomials  $P_{G_i}^M$.
\end{prop}
\begin{proof} We will prove the statement by induction on \(m\). 
We notice that if   $m=1$ the statement is true by the Lemma \ref{casounavariabile}   (the ideal $I_1$ coincides with the ideal \(I\) in  the lemma).

Let then  \(m\geq 2\) and let us consider the ideal $I_{m-1}\subset B[t_1,\ldots ,t_{m-1}]$ which by the inductive hypothesis does not depend on the choice of the polynomials $P_{G_i}^M$'s (where \(i<m\)). We will denote by \(I'_{m-1}\) its extension to   $B[t_1,\ldots ,t_m]$.

The polynomials   $F(i,A)$ belong to  $I'_{m-1}$ unless  $m\in A$ or $i=m$.
In the latter case  $A=\emptyset$ and  the same proof as in Lemma \ref{casounavariabile} implies that the ideal does not depend on the choice of  the polynomial \(P_{G_m}\).

In the first case ($m\in A$), if we consider two liftings 
$P^M_{G_i}$ and  $Q^M_{G_i}$ we notice that the restriction  of their difference $P^M_{G_i}-Q^M_{G_i}$  to  $H^*( M)[t]$ has constant term equal to 0, while the restriction to $H^*( G_i)[t]$ is 0. 

Let  $z$ be equal to $P^M_{G_i}(0)-Q^M_{G_i}(0)$. By construction  $z$ belongs to the ideal generated by the   $c_r$'s such that  $r$ does  not belong to \(V_{\Lambda_{M}}\), that is to say, \(\langle r, \cdot \rangle\) does not vanish on \(\Lambda_M\).  Now we observe that the lattice  $\Gamma= \sum_{j\in A} \Lambda_{G_j}$ has finite index  in $\Lambda_M$. If \(\langle r, \cdot \rangle\) vanished on   $\Lambda_{G_j}$ for every  $j\in A$ then it would vanish on   $\Gamma $  and therefore  on $\Lambda_M$. 

It follows that  if $r$  does  not belong to \(V_{\Lambda_{M}}\) then it exists  $j\in A$ such that  $r$ does  not belong to \(V_{\Lambda_{G_j}}\). 
This implies   that $\prod_{j\in A}t_{j}z$ belongs to the ideal generated by the monomials  in  (1) . To conclude it is sufficient to notice that the coefficients of 
$P^M_{G_i}(\sum_{h\in B_i}-t_{h})-Q^M_{G_i}(\sum_{h\in B_i}-t_{h})-z$ belong to the ideal generated by the $c_r$'s such that  $r \notin V_{\Lambda_{G_i}}$, and therefore, for the same reasoning as above, 
$P^M_{G_i}(\sum_{h\in B_i}-t_{h})-Q^M_{G_i}(\sum_{h\in B_i}-t_{h})-z$ belongs to \(I'_{m-1}\).

\end{proof}

\section{Presentation of the cohomology ring}
\label{seccohomology}

Let us  consider a toric arrangement $\mathcal A$ in the torus $T$.
As recalled in Section \ref{pipo}, let $X=X_\Delta$ be a good  toric variety associated to the chosen toric arrangement, and let us consider the arrangement $\mathcal L'$ of subvarieties in $X$.

Fix now a well connected building set  \(\G=\{G_1,...,G_m\}\)  for  $\mathcal L'$,  ordered in such a way  that   if \(G_i\subsetneq  G_j\)  then \(i<j \).

Our goal is to describe   the cohomology ring \(H^*(Y(X,\G),\Z)\) by generators and relations.
For any pair  \((G, M)\in \mathcal L'\times \mathcal L'\) with    $G\subset M$,   we fix a polynomial  \(P_G^M\in H^*(X,\Z)[t]=B[t]\) as explained in the preceding section. We also fix the polynomials \(P_G^X\in B[t]\).  This means in particular that we have fixed a choice  for the polynomials \(F(i,A)\in B[t_1,...,t_m]\).
Then we can state our main theorem:

\begin{teo}
\label{teopresentazionecoomologia}
The cohomology ring \(H^*(Y(X,\G), \Z)\) is isomorphic to the polynomial ring  $B[t_1,\ldots ,t_m]$  modulo the ideal $J_m$ generated by the following elements:
\begin{enumerate}[1.]
\item The products  $t_ic_{r}$, with   \(i\in \{1,...,m\}\),  for every  ray $r\in \mathcal R $  that does not belong to \(V_{\Lambda_{G_i}}\).

\item The polynomials \(F(i,A)\), for every pair  $(i,A)$ with   \(i\in \{1,...,m\}\) and    \(A\subset \{1,...,m\}\) such that if \(j\in A\) then \( G_i\subsetneq G_j\), and  for the pairs \((0,A)\) where  $A$ is such that  \(\bigcap_{j\in A}G_j=\emptyset \).
\end{enumerate}
The isomorphism is given by sending, for every \(i=1,...,m\), \(t_i\) to the pull back under the projection $\pi_i:Y(X,\G)\to X_i=Bl_{\tilde G_i}X_{i-1}$ of the  class of the exceptional divisor in $X_i$.
\end{teo}
\begin{proof}As a preliminary remark, let us observe  that the ideal generated by the relations in the statement of the theorem, according to Proposition \ref{propindependent},  does not depend on the choice of the polynomials \(F(i,A)\). In this proof we will use the following notation: if a polynomial \(g\) is another choice for \(F(i,A)\) we will write \(g\sim F(i,A)\).

The proof of the theorem is by induction on the cardinality \(m\) of \(\G\). The case when \(m=0\) is obvious.

Let us now suppose that the statement of the theorem  is true for any  projective model   constructed starting from a toric arrangement \(\A'\) in a torus \(T'\), and then choosing a good toric variety for \(\A'\) and a well connected building set with cardinality \(\leq m-1\). 

 In particular it is true for the  the variety \(Y(X,\G_{m-1})\). Let us use the notation of Section \ref{wellconnectedbuilding} and in particular set \(Y(X,\G_{m-1})=X_{m-1}\) and $Z=G_m$.  Now, in order to get \(Y(X,\G)\) we have to blowup \(X_{m-1}\) along the proper transform of \(Z\) which by Proposition \ref {zetat} is equal to  \(Z_u\).

Since \(\G\) is a building set for $\mathcal L'$, we know that  \(Z\) is the closure of a layer \({\mathcal K}_{\Gamma,\phi}\subset T\),  which  is  a coset with respect to the subtorus 
$H=\cap_{\chi\in \Gamma}Ker(x_\chi)$ of \(T\). Up to translation, we identify \({\mathcal K}_{\Gamma,\phi}\subset T\) with $H$.

Under this identification we get the  arrangement \(\A_H\) in \(H\), given by the  connected components of the intersections \( A\cap {\mathcal K}_{\Gamma,\phi}\) for every \(A\in \A\). Notice that \(X^*(H)=X^*(T)/\Gamma\).

We know that $Z$ is the $H$-variety associated to the fan \(\Delta_H\), consisting of those cones in $\Delta$ which lie in $V_{\Lambda_Z}$. From this it is immediate to check that $Z$ is    a good toric variety for \(\A_H\). If we  denote by $\mathcal L'_H$ its corresponding arrangement of subvarieties,  we also have, by Proposition  \ref{lemmarrangiamentoindotto}, that  \(\mathcal H\) is a well connected building set for $\mathcal L'_H$. 
Thus since $u\leq m-1$, we can also assume that our result holds for $H^*(Z_u,\mathbb Z)$.

To be more precise we can assume that the cohomology ring \(H^*(X_{m-1}, \Z)\) is isomorphic to the polynomial ring  $B[t_1,\ldots ,t_{m-1}]$  modulo the ideal $J_{m-1}$ generated by \begin{enumerate}
\item The products  $t_ic_{r}$, with   \(i\in \{1,...,m-1\}\),  for every  ray $r\in \mathcal R $  that does not belong to \(V_{\Lambda_{G_i}}\).

\item The polynomials \(F(i,A)\), for every pair  $(i,A)$ with   \(i\in \{1,...,m-1\}\) and    \(A\subset \{1,...,m-1\}\) such that if \(j\in A\) then \( G_i\subsetneq G_j\), and  for the pairs \((0,A)\) where  $A$ is such that  \(\bigcap_{j\in A}G_j=\emptyset \).\end{enumerate}
The isomorphism is given by sending, for every \(i=1,...,m-1\), \(t_i\) to the pull back under the projection $\pi_i:X_{m-1}\to X_i=Bl_{\tilde G_i}X_{i-1}$ of the  class of the exceptional divisor in $X_i$.

As far as  $Z_u$ is concerned we need to fix some notation.

Following what we have done for $X$ and $\mathcal G$,
for every pair  $(i,A)$ with   \(i\in \{1,...,u\}\), and    \(A\subset \{1,...,u\}\) such that if \(j\in A\) then \( H_i\subsetneq H_j\),   we  define  the polynomial  $F_Z(i,A)$ in  $H^*(Z,\mathbb Z)[z_1,\ldots ,z_u]$, as follows.

 We   consider the set  \(C_i=\{h\: | \: H_h\subseteq H_i\}\),
and we denote by \(M\) the unique connected component of  \(\bigcap_{j\in A}H_j \)  that contains  $H_i$ (if   $A=\emptyset$ we put  $M=Z$). Then we restrict the polynomials  $P^M_{H_i}$ 
to  $H^*(Z,\mathbb Z)[t]$ and we denote these restrictions  by $P^M_{H_i,Z}$.
We put:
  \[ F_Z(i,A)=P^M_{H_i,Z}(\sum_{h\in C_i}-z_{h})\prod_{j\in A}z_{j}. \]

 As before
we  include  the pairs \((0,A)\)  with  
\(\bigcap_{j\in A}H_j=\emptyset \), and we set:  
\[ F_Z(0,A)=\prod_{j\in A}z_{j}.\]

Then, setting   $B'=H^*(Z,\mathbb Z)$, we  can  assume that  cohomology ring \(H^*(Z_u, \Z)\) is isomorphic to the polynomial ring $B'[z_1,\ldots ,z_{u}]$  modulo the ideal $S$ generated by \begin{enumerate}
\item The products  $z_ic_{r}$, with   \(i\in \{1,...,u\}\),  for every  ray $r\in \Delta $  that does not belong to \(V_{\Lambda_{H_i}}\).
\item The polynomials \(F_Z(i,A)\), for every pair  $(i,A)$ with   \(i\in \{1,...,u\}\) and    \(A\subset \{1,...,u\}\) such that if \(j\in A\) then \( H_i\subsetneq H_j\), and  for the pairs \((0,A)\) where  $A$ is such that  \(\bigcap_{j\in A}H_j=\emptyset \).

\end{enumerate}
The isomorphism is given by sending, for every \(i=1,...,u\), \(z_i\) to the pull back under the projection $\pi_i:Z_{u}\to Z_i=Bl_{\tilde H_i}Z_{i-1}$ of the  class of the exceptional divisor in $Z_i$.

Let us now consider the homomorphisms
$$j^*:H^*(X_{m-1}, \Z)\to H^*(Z_{u}, \Z)$$
and 
$$\iota^*:H^*(X, \Z)\to H^*(Z, \Z)$$
induced by  the respective inclusions.
We now remark  that by the discussion in the proof of Proposition  \ref{zetat}, we get that, denoting by 
$[t_j]$ (resp. $[z_i]$) the image of $t_j$ (resp. $z_i$) in  $H^*(X_{m-1}, \Z)$ (resp. $H^*(Z_{u}, \Z)$),
$$ j^*([t_i])=\begin{cases} 0 \ \text{if\ }j\neq s_i\\ [z_i]  \ \text{if\ }j= s_i\end{cases}$$
From this we deduce that $j^*$ is surjective and, if we define 
$$f:B[t_1,\ldots ,t_{m-1}]\to B'[z_1,\ldots ,z_{u}]$$

$$f(a)=\iota^*(a) \ \text {if\ } a\in B$$ $$ f(t_i)=\begin{cases} 0 \ \text{if\ }j\neq s_i\\ z_i  \ \text{if\ }j= s_i\end{cases},$$
we get a commutative diagram:
\[\begin{tikzcd}
B[t_1,\ldots ,t_{m-1}] \arrow{r}{f} \arrow{d}{p} & B'[z_1,\ldots ,z_{u}] \arrow{d}{q} \\
H^*(X_{m-1}, \Z)\arrow{r}{j^*} & H^*(Z_{u}, \Z)
\end{tikzcd}
\]
where $p$ and $q$ are the quotient maps.
At this point we can apply Theorem \ref{teoKeel}. 

We  deduce that
$H^*(Y(X,\mathcal G,\mathbb Z)$ is isomorphic to    $B[t_1,\ldots ,t_m]/L$  where  the ideal
$L=(J_{m-1}, 
t_mker (q\circ f), P_{ Z_u}(-t_m))$.

 In order to proceed, we need an explicit description of the generators for the ideal $ker (q\circ f)$. From the definition of $f$ and our description of the relations for $H^*(Z_u,\mathbb Z)$ we deduce that $ker \ q\circ f$ is generated by 
\begin{enumerate}
\item The elements $c_r$, for every  ray $r\in \Delta$  which does not belong to \(V_{\Lambda_{Z}}\).
\item The elements $t_j$, with $1\leq j\leq m-1$, $j\notin \{s_1,...,s_u\}$.
\item The products  $t_{s_i}c_r$, with   \(i\in \{1,...,u\}\),  for every  ray $r\in \Delta $  that does not belong to \(V_{\Lambda_{H_i}}\).

\item For every $(s_i,A)$ with   \(i\in \{1,...,u\}\) and    \(A\subset \{s_1,...,s_u\}\) such that if \(s_j\in A\) then \( H_i\subsetneq H_j\), the elements  \[\check{F}(s_i, A):=P^M_{H_{i}}(\sum_{h\in B_{s_i}}-t_{h})\prod_{s_j\in A}t_{s_j},\]
where  $M$ is the connected component of $\cap_{s_j\in A}H_j$ that contains \(H_i\), if $A\neq \emptyset$, $G_m$ otherwise. 

Indeed \(f(\check{F}(s_i,A))= F_Z(i,\overline A)\) 
where \(\overline A=\{j\: | \: s_j\in A\}\)  and therefore it belongs to \(ker \ q\).

\item The polynomials \(F(0,A)\) for the pairs \((0,A)\) where  $A\subset \{s_1,...,s_u\}$ is such that  \(\bigcap_{s_j\in A}H_j=\emptyset \).  

\end{enumerate}
Notice that the elements in (1)  and (2) generate $ker f$.

We want to show that \(L\) is equal to the  ideal \(J_m\) generated by the elements described in the statement of the theorem.
Let us first show that \(J_m\subset L\).

The generators of \(J_m\) that do not contain \(t_m\) belong to \(J_{m-1}\) and therefore to \(L\). 

A generator of the form   $t_mc_{r}$, for a  ray $r\in \mathcal R $  that does not belong to \(V_{\Lambda_{Z}}\) clearly lies in $t_mker (q\circ f)$. 

Take a  generator of the form \(F(j,A)\) with $m\in A$ and $j>0$. Set $A'=A\setminus \{m\}$. Then
\[ F(j,A)=t_m(P^M_{G_j}(\sum_{h\in B_j}-t_{h})\prod_{\nu \in A'}t_{\nu}). \]
If there is a $\nu\in A'$ such that $\nu$ is not one of the $ s_i$'s, then 
$$P^M_{G_j}(\sum_{h\in B_j}-t_{h})\prod_{\nu \in A'}t_{\nu}\in ker f$$ and we are done. 

Otherwise, set $\overline{A}=\{i|s_i\in A'\}$.
Notice that since $G_j\subset Z$ necessarily $G_j=G_{s_i}$ for some $1\leq i\leq t$,
and $B_j=\{h|h=s_k, H_k \subseteq H_i\}$. We deduce that 
$$f(P^M_{G_j}(\sum_{h\in B_j}-t_{h})\prod_{\nu \in A'}t_{\nu})= F_Z(i, \overline A)$$
and therefore it belongs to \( ker\, q\). Finally consider $F(0,A)=\prod_{\nu\in A}t_\nu$, with $m\in A$.  
If there is a $\nu\in A'=A\setminus \{m\}$ such that $\nu$ is not one of the $ s_i$'s, then $\prod_{\nu\in A}t_\nu\in ker f$.  Otherwise, set $\overline A=\{i|s_i\in A'\}$. We deduce that 
$$f(F(0,A))=F_Z(0, \overline A)\in ker\, q,\ \ 
\text{since \ }\bigcap_{i\in\overline A}H_i=\bigcap_{\nu\in A'}(G_\nu\cap Z)=\bigcap_{\nu\in A}G_\nu=\emptyset.$$

Finally in order to show that also $F(m,\emptyset)\in L$ we need the following well known
\begin{lemma}\label{remfultonintersection}
Let \(W_1\subset W_3\) and \(W_2\subset W_3\) be   regular imbeddings with normal bundles $N_{ W_1} W_3$ and $N_{ W_2} W_3$ . Set \(\widetilde W_3=BL_{W_1}W_3\) and let \(\widetilde W_2\) denote the dominant transform of $W_2$. 

Then the canonical imbedding $\widetilde W_2\subset \widetilde W_3$ is regular and 
denoting by $\pi$  the projection  from $\widetilde W_2$ to \(W_2\),   
\begin{enumerate}[1.]
\item If \(W_1\subset W_2\)
\[N_{\widetilde W_2}\widetilde W_3\cong \pi^* N_{ W_2} W_3\otimes \mathcal O(-E)\]
where  \(E \) is the exceptional divisor on \(\widetilde W_2\). 
\item If the intersection of $W_1$ and $W_2$ is transversal,
\[ N_{\widetilde W_2}\widetilde W_3\cong \pi^* N_{ W_2} W_3.\]
\end{enumerate}
\end{lemma}\begin{proof} For 1. see \cite{Fultonintersection}, B.6.10. The second part is easy.\end{proof}

By  repeated use of this lemma,   we easily get that 
$$F(m,\emptyset)=P_{ Z}(-\sum_{h\in B_m}t_h)=P_{ Z_u}(-t_m)\in L$$ so that indeed $L\supseteq J_m.$
To finish, we need to see that $L\subseteq J_m.$

We first observe  that   \(J_{m-1}\subset J_m\). Furthermore, we have already seen that 
$P_{ Z_u}(-t_m)=F(m,\emptyset)\in J_m$. It follows that we need to concentrate on the generators of
$ker (q\circ f)$ multiplied by $t_m$.
Following the list given above we consider:
\begin{enumerate}
\item The elements \(t_m c_r\), for every  ray $r\in \Delta$  which does not belong to \(V_{\Lambda_{Z}}\).
These are also generators of  \(J_m\) and there is nothing to prove.
\item  The products $t_mt_j$, with $1\leq j\leq m-1$, and $j$  is not one of the $ s_i$'s. We notice that \(G_j\cap G_m\) is either empty, and therefore $t_mt_j=F(0,\{j,m\})\in J_m\), or each connected component of \(G_j\cap G_m\)  belongs to \(\G\). Let \(G_h\) be one of these components. Then  the generator \(F(h, \{j,m\})\) of \(J_m\) is equal to $t_mt_j$ since \(F(h, \{j,m\})=t_mt_jP_{G_h}^{G_h}\) and \(P_{G_h}^{G_h}=1 \). This finishes the proof that $t_mker f\subset L$.
\item The products $t_mt_{s_i}c_r$, with   \(i\in \{1,...,u\}\),  for every  ray $r\in \Delta $  that does not belong to \(V_{\Lambda_{H_i}}\). There are two possibilities.  If \(H_i=G_{s_i}\) then \(V_{\Lambda_{H_i}}=V_{\Lambda_{G_{s_i}}}\) and $t_{s_i}c_r$ is a generator of \(J_{m-1}\). 

If \(H_i\) is the transversal intersection of \(Z\) and \(G_{s_i}\) then \(V_{\Lambda_{H_i}}=V_{\Lambda_{G_{s_i}}}\cap V_{\Lambda_{Z}}\). Therefore if \(r\) does not belong to \(V_{\Lambda_{H_i}}\) either it does  not belong to \(V_{\Lambda_{Z}}\), and then  \(t_mc_r\) is a generator   of \(J_m\) that  has already been considered in (1), or it does  not belong to \(V_{\Lambda_{G_{s_i}}}\) and $t_{s_i}c_r$ is a generator of \(J_{m-1}\). 

\item  The elements \(t_m \check{F}(s_i,A)\), for every pair  $(s_i,A)$ with   \(i\in \{1,...,u\}\) and    \(A\subset \{s_1,...,s_u\}\) such that if \(s_j\in A\) then \( H_i\subsetneq H_j\). 

If $G_{s_i}\subset G_m$, that is $G_{s_i}=H_i$, then, since $M$ is the connected component of $G_m\cap(\cap_{s_j\in A}G_{s_j})$ containing $H_i$, it is clear that   \[t_m \check{F}(s_i,A)=F(s_i, A\cup \{m\})\in J_m.\]

Otherwise  \(H_i\) does not belong to \(\G\) and it is the transversal intersection of \(G_{s_i}\) and \(G_m\) (see Proposition \ref{lemmaintersezione}), that are   its \(\G\) factors. 
If $A=\emptyset$, we observe that \(P_{G_{s_i}}^X\) is  a valid choice for \(P_{H_i}^Z\)  so 
\(\check{F}(s_i,\emptyset)\sim F(s_i,\emptyset)\). 

Therefore $\check{F}(s_i,\emptyset) \in J_{m-1}\) and $t_m\check{F}(s_i,\emptyset)\in J_m$.

Assume now $A\neq \emptyset$. We claim that, denoting by $M'$   the connected component of the intersection
 $\cap_{s_j\in A}G_{s_j}$ containing  $G_{s_i}$, $M$ is the transversal intersection of $M'$ and $G_m$.
 
Take any $t$ such that \( H_i\subseteq H_t\). 
Then if $G_{s_t}=H_t\subset G_m$, since $G_m$ is a $\G$ factor of
$H_i$ this would imply $G_{s_t}=G_m$ a contradiction. 

  We deduce that $G_{s_j}\nsubseteq G_m$ for all $s_j\in A$,
and furthermore $M\notin \G$, since  otherwise  $G_m=M\subset  G_{s_j}$.

A $\G$ factor of $M'$ is contained in at least one of the $G_{s_j}$, $s_j\in A$. In particular none of these $\G$ factors contains $G_m$. 
 Furthermore since  $G_m$ is a $\G$ factor of  $H_i$ it is also a $\G$ factor of $M$.

It follows that we can apply  Corollary \ref{cortransversale}.2 and we conclude that $M$ is  the transversal intersection of $M'$ and $ G_m$ as desired. 
Thus, reasoning as above,  
we observe that \(P_{G_{s_i}}^{M'}\) is  a valid choice for \(P_{H_i}^M\)  so 
\(\check{F}(s_i, A)\sim F(s_i,A)\).
Therefore $\check{F}(s_i,A) \in J_{m-1}\) and $t_m\check{F}(s_i,A)\in J_m$.

 \item The products $t_mF(0,A)=t_m(\prod_{s_i\in A}t_{s_i})$ for  \(A\subset \{s_1,...,s_u\}\) such that  $\cap_{s_i\in A}H_i=\emptyset$. In this case
 $$G_m\cap (\bigcap _{i\in A}G_{s_i})=\bigcap_{i\in A}H_i=\emptyset$$
 so that $t_m\prod_{i\in A}t_{s_i}=F(0, A\cup\{m\})\in J_m$.
\end{enumerate}

Putting everything together we have shown that $L\subset J_m$ so that $L=J_m$ and our claim is proved.
\end{proof}
\section{A way to choose the polynomials \(P_G^M\)}
\label{subsecsceltapolinomi}

Let us use the same notations (\(\A,\Delta, X=X_\Delta,...\)) as in the preceding sections.
We want to show an explicit choice of the polynomials $P_G^M\in H^*(X,\Z)[t]=B[t]$, and therefore of the polynomials \(F(i,A)\) that appear in Theorem \ref{teopresentazionecoomologia}

Let us  consider two elements \(G,M\in \mathcal L'\) with   $G\subset M$. 
We can choose  a basis \(B_{\Lambda_G}=\{\beta_1,...,\beta_s\}\)   of   $\Lambda_G$ such that the first \(k\) elements (\(k<s\))  are a basis of  \(\Lambda_M\).


We recall that the irreducible divisors in the boundary of \(X\) are in correspondence with the rays   of the fan \(\Delta\). 

In particular, let us consider a maximal cone \(\sigma\) in \(\Delta\), whose one dimensional faces are  generated by the  rays  \(r_1,...,r_n\) (a basis of the lattice  \(hom_\Z(X^*(T), \Z)\)),   and let us denote as usual their corresponding divisors by  \(D_{r_1},...,D_{r_n}\).

The subvariety  \(G= {\overline {\mathcal K}}_{\Lambda_G, \phi}\) of \(X\) has  the following local defining equations in the chart associated to \(\sigma\):
\[  \left\{
\begin{array}{c}
  z_1^{\langle \beta_1,r_1\rangle }\cdots z_n^{\langle \beta_1,r_n\rangle }=\phi(\beta_1)  \\
  z_1^{\langle \beta_2,r_1\rangle }\cdots z_n^{\langle \beta_2,r_n\rangle }=\phi(\beta_2)   \\
  ...\\
  ...\\
  z_1^{\langle \beta_s,r_1\rangle }\cdots z_n^{\langle \beta_s,r_n\rangle }=\phi(\beta_s)
\end{array}
 \right.\]

Therefore the subvariety  \( G\) is described as the intersection of \(s\) divisors. The divisor \(D(\beta_j)\) corresponding to \(\beta_j\) has  a local function  with  poles of order  \(- min (0, \langle \beta_j ,r_i\rangle)  \) along the divisor \(D_{r_i}\), for every \(i=1,..., n\).
This implies that in \(Pic(X)\) we have the following relation:
\begin{equation}\label{relazione}
[D(\beta_j)]+\sum_{r}min(0, \langle \beta_j,r\rangle )[D_r]=0
\end{equation}
where \(r\) varies in the set \(\mathcal R\) of  all the rays  of \(\Delta\).

Therefore the   polynomial  in \(H^*(X,\Z)[t]=B[t]\)  \[P_G^X=\prod_{j=1}^s(t-\sum_{r\in \mathcal R}min(0,\langle \beta_j, r\rangle )c_r)\]  where  \(c_r\) is the class of the divisor  \(D_r\), is a good  lifting  of the Chern polynomial of  $N_{ G}X$.

At the same way, the   polynomial    in \(B[t]\) \[P_M^X=\prod_{j=1}^k(t-\sum_{r\in \mathcal R}min(0,\langle \beta_j, r\rangle )c_r)\] is a good  lifting  of the Chern polynomial of  $N_{M}X$. 
This implies that the polynomial \[\frac{P^X_G}{P^X_M}=\prod_{j=k+1}^s(t-\sum_{r\in \mathcal R}min(0,\langle \beta_j, r\rangle )c_r) \]
restricted to \(H^*( M, \Z)[t]\) is   a good lifting  of the Chern polynomial of   $N_{ G}( M)$, i.e. it is a choice for the polynomial \(P_G^M\)  as requested in   Section \ref{Chern1}. 

\section{The cohomology of the strata}
\label{seccohomologystrata}
Let us consider, with  the same notation as before (\(\A,\Delta, \mathcal R, X=X_\Delta,\elle,\elle'\)), a well connected building set \(\G=\{G_1,...,G_m\}\) for \(\elle'\). As we know from Section \ref{pipo}, the models \(Y(X,\G)\) and \(Y(X,\G^+)\) are isomorphic. As in Proposition \ref{ilg+}, we set  $\mathcal G^+=\G \cup \{D_r\}_{r\in \mathcal R}\), and  for any \(G\in \G^+\) we denote by \(D_G\)  its   corresponding divisor in  \(Y=Y(X,\G^+)\).

In this section we are going to generalize  our main result and explain how to compute the cohomology ring for any variety
$Y_{\mathcal S}=\bigcap_{G\in \mathcal S}D_G$ for any subset $\mathcal S\in\mathcal G^+$. Notice that if $\mathcal S$ is not ($\mathcal G^+$)-nested, $Y_{\mathcal S}=\emptyset$, so that we are going to assume that $\mathcal S$  is nested.

We set $\mathcal T_{\mathcal S}=\mathcal S\cap \mathcal G$ and
$\mathcal D_{\mathcal S}=\mathcal S\cap  \{D_r\}_{r\in \mathcal R}\), so that
\(\mathcal S$ is the disjoint union of $\mathcal T_{\mathcal S}$ and $\mathcal D_{\mathcal S}$. Remark that, since $\mathcal S$ is nested, the rays $\mathcal R_{\mathcal S}=\{r|D_r\in \mathcal D_{\mathcal S}\}$ span a cone in the fan $\Delta$.

Fix a pair  $(i,A)$ with   \(i\in \{1,...,m\}\), and    \(A\subset \{1,...,m\}\) such that if \(j\in A\) then \( G_i\subsetneq G_j\). Set $\mathcal S_i=\{h|G_h\in \mathcal S\ \text{and}\ G_h\supsetneq G_i\}$ and  consider the set  \(B_i=\{h\: | \: G_h\subseteq G_i\}\).
Denote by \(M=M_{\mathcal S}\) the unique connected component of  \(\bigcap_{j\in A\cup \mathcal S_i}G_j \)  that contains  $G_i$ (if   $A\cup \mathcal S_i=\emptyset$ we put  $M=X$). Then, after choosing   all the polynomials $P^M_{G_i}$ as explained in the previous sections, we set:
  \[ F_{\mathcal S}(i,A)=P^M_{G_i}(\sum_{h\in B_i}-t_{h})\prod_{j\in A}t_{j} \]
  We also set \(F_{\mathcal S}(0, A)=F(0, A)\).
  We have
\begin{teo}\label{teopresentazionecoomologiastrati} For any nested set $\mathcal S\subset \mathcal G^+$, the cohomology ring $H^*(Y_{\mathcal S},\mathbb Z)$ is isomorphic to  the polynomial ring  $B[t_1,\ldots ,t_m]$  modulo the ideal $J_m(\mathcal S)$ generated by the following elements:
\begin{enumerate}[1.]
\item The classes $c_r\in B$ for any ray $r$ such that $\{r\}\cup R_{\mathcal S}$ 
does not span a cone in the fan $\Delta$.
\item The products  $t_ic_{r}$, with   \(i\in \{1,...,m\}\),  for every  ray $r\in \mathcal R $  that does not belong to \(V_{\Lambda_{G_i}}\).

\item The polynomials \(F_{\mathcal S}(i,A)\), for every pair  $(i,A)$ with   \(i\in \{1,...,m\}\) and    \(A\subset \{1,...,m\}\) such that if \(j\in A\) then \( G_i\subsetneq G_j\), and  for the pairs \((0,A)\) where  $A$ is such that  \[\left( \bigcap_{j\in A}G_j\right )\cap \left( \bigcap_{H\in {\mathcal S}}H \right)=\emptyset \]
\end{enumerate}
The image in $H^*(Y_{\mathcal S},\mathbb Z)$ of the classes $c_r$ and $t_j$ is just the restriction of the corresponding classes in $H^*(Y(X, \mathcal G^+),\mathbb Z).$
\end{teo}
\begin{proof} As in the proof of Theorem \ref{teopresentazionecoomologia} we proceed by induction on $m$.
The case $m=0$ follows from the well know computation of the cohomology of stable subvarieties in a complete smooth toric variety (\cite{fultontoric}). 
So we take $\mathcal G_{m-1}^+=\mathcal G^+\setminus G_m$ which, by Theorem \ref{ilg+}, is a building set.
Furthermore  we remark that the nested sets in $\mathcal G_{m-1}^+$ coincide with the nested sets in $\mathcal G^+$ not containing $G_m$. 

We set for any $G\in \mathcal G^+_{m-1}$, $D'_G$ equal to the divisor corresponding to $G$ in $Y'=Y(\mathcal G_{m-1},X)$ and for $\mathcal S$ nested in $\mathcal G^+_{m-1}$, $Y'_{\mathcal S}=\cap_{G\in \mathcal S}D'_G$.

Let us take a nested set $\mathcal S$ for $\mathcal G^+$
and, as usual,  put \(G_m=Z\). 

Assume $G_m\notin \mathcal S$. 
If $\mathcal S\cup \{G_m\}$ is not nested, then $Y'_{\mathcal S}\cap \tilde Z=\emptyset$, so that 
$Y'_{\mathcal S}=Y_{\mathcal S}$. In particular $t_m$ is in the kernel of the restriction map $H^*(Y(\mathcal G, X))\to H^*(Y_{\mathcal S}) $.
Now $t_m=F(0, \{m\})$ is one of our relations and all the other relations different from $F_{\mathcal S}(m,\emptyset)$ either are divisible by $t_m$ or they already appear among the relations for $H^*(Y'_{\mathcal S})$. As for $F_{\mathcal S}(m,\emptyset)$,  this coincides with  $F_{}(m,\emptyset)$, therefore  it is already equal to $0$ in $H^*(Y(X,\mathcal G))$ so  there is nothing to prove.

If $\mathcal S\cup \{G_m\}$ is  nested then the intersection $N=Y'_{\mathcal S}\cap \tilde Z$ is transversal, so that $Y_{\mathcal S}$ is the blow up of $Y'_{\mathcal S}$ along $N$. Now $N$ is just
the transversal intersection of the divisors $D'_{G_i}\cap \tilde Z$ in $\tilde Z$ then again we can use our inductive hypothesis  exactly as in the proof of Theorem \ref{teopresentazionecoomologia}.

Finally if $G_m\in \mathcal S$, setting $\mathcal S'=\mathcal S\setminus \{G_m\}$, we deduce that $Y_{\mathcal S}$ is  the blow up of $Y'_{\mathcal S'}$ along the (necessarily transversal) intersection $Y'_{\mathcal S'}\cap \tilde Z$. Thus again everything follows from our inductive assumption and the nature of the relations.

\begin{rem}
We notice that the arguments used in the proof above and in the proof of Theorem \ref{teopresentazionecoomologia} can be applied almost verbatim to the case of projective wonderful models of a subspace arrangement in \(\pp (\C^n)\). 
Everything in this case is simpler: all the building sets are well connected,  the polynomials \(P_G^M(t)\) are powers of \(t\) and the initial projective variety is \(\pp (\C^n)\). One finally gets, with a shorter proof,  the same presentation by generators and relations of  Theorem 5.2 in \cite{DCP1}.

\end{rem}

\end{proof}
\addcontentsline{toc}{section}{References}
\bibliographystyle{acm}
\bibliography{Bibliogpre} 

\begin{thebibliography}{10}

\bibitem{adiprasitokatzhuh}
{\sc Adiprasito, K., Huh, J., and Katz, E.}
\newblock Hodge theory for combinatorial geometries.
\newblock {\em arXiv:1511.02888\/} (2015).

\bibitem{callegarogaiffilochak}
{\sc Callegaro, F., Gaiffi, G., and Lochak, P.}
\newblock Divisorial inertia and central elements in braid groups.
\newblock {\em Journal of Algebra 457\/} (2016), 26 -- 44.

\bibitem{DCG}
{\sc De~Concini, C., and Gaiffi, G.}
\newblock Projective wonderful models for toric arrangements.
\newblock {\em Advances in Mathematics
  https://doi.org/10.1016/j.aim.2017.06.019\/} (2017).

\bibitem{DCP2}
{\sc De~Concini, C., and Procesi, C.}
\newblock Hyperplane arrangements and holonomy equations.
\newblock {\em Selecta Mathematica 1\/} (1995), 495--535.

\bibitem{DCP1}
{\sc De~Concini, C., and Procesi, C.}
\newblock Wonderful models of subspace arrangements.
\newblock {\em Selecta Mathematica 1\/} (1995), 459--494.

\bibitem{DS}
{\sc {Denham}, G.~C., and {Suciu}, A.~I.}
\newblock {Local systems on arrangements of smooth, complex algebraic
  hypersurfaces}.
\newblock {\em ArXiv e-prints\/} (June 2017).

\bibitem{etihenkamrai}
{\sc Etingof, P., Henriques, A., Kamnitzer, J., and Rains, E.}
\newblock The cohomology ring of the real locus of the moduli space of stable
  curves of genus 0 with marked points.
\newblock {\em Annals of Math. 171\/} (2010), 731--777.

\bibitem{feichtner}
{\sc Feichtner, E.}
\newblock De {C}oncini-{P}rocesi arrangement models - a discrete geometer's
  point of view.
\newblock {\em Combinatorial and Computational Geometry, J.E. Goodman, J. Pach,
  E. Welzl, eds; MSRI Publications 52, Cambridge University Press\/} (2005),
  333--360.

\bibitem{feichtneryuz}
{\sc Feichtner, E., and Yuzvinsky, S.}
\newblock Chow rings of toric varieties defined by atomic lattices.
\newblock {\em Invent. math. 155}, 3 (2004), 515--536.

\bibitem{fultontoric}
{\sc Fulton, W.}
\newblock {\em Introduction to toric varieties}, vol.~131 of {\em Annals of
  Mathematics Studies}.
\newblock Princeton University Press, Princeton, NJ, 1993.
\newblock The William H. Roever Lectures in Geometry.

\bibitem{Fultonintersection}
{\sc Fulton, W.}
\newblock {\em Intersection theory}, second~ed., vol.~2 of {\em Ergebnisse der
  Mathematik und ihrer Grenzgebiete. 3. Folge. A Series of Modern Surveys in
  Mathematics [Results in Mathematics and Related Areas. 3rd Series. A Series
  of Modern Surveys in Mathematics]}.
\newblock Springer-Verlag, Berlin, 1998.

\bibitem{GaiffiBlowups}
{\sc Gaiffi, G.}
\newblock Blow ups and cohomology bases for {D}e {C}oncini-{P}rocesi models of
  subspace arrangements.
\newblock {\em Selecta Mathematica 3\/} (1997), 315--333.

\bibitem{gaiffipermutonestoedra}
{\sc Gaiffi, G.}
\newblock Permutonestohedra.
\newblock {\em Journal of Algebraic Combinatorics 41\/} (2015), 125--155.

\bibitem{hendersonwreath}
{\sc Henderson, A.}
\newblock Representations of wreath products on cohomology of {D}e
  {C}oncini-{P}rocesi compactifications.
\newblock {\em Int. Math. Res. Not.}, 20 (2004), 983--1021.

\bibitem{keel}
{\sc Keel, S.}
\newblock Intersection theory of moduli space of stable {$n$}-pointed curves of
  genus zero.
\newblock {\em Trans. Amer. Math. Soc. 330}, 2 (1992), 545--574.

\bibitem{li2}
{\sc Li, L.}
\newblock Chow motive of {F}ulton-{M}ac{P}herson configuration spaces and
  wonderful compactifications.
\newblock {\em Michigan Math. J. 58}, 2 (2009), 565--598.

\bibitem{li}
{\sc Li, L.}
\newblock Wonderful compactification of an arrangement of subvarieties.
\newblock {\em Michigan Math. J. 58\/} (2009), 535--563.

\bibitem{mociwonderful}
{\sc Moci, L.}
\newblock Wonderful models for toric arrangements.
\newblock {\em Int. Math. Res. Not.}, 1 (2012), 213--238.

\bibitem{rains}
{\sc Rains, E.}
\newblock The homology of real subspace arrangements.
\newblock {\em J Topology 3 (4)\/} (2010), 786--818.

\bibitem{yuzBasi}
{\sc Yuzvinsky, S.}
\newblock Cohomology bases for {D}e {C}oncini-{P}rocesi models of hyperplane
  arrangements and sums over trees.
\newblock {\em Invent. math. 127\/} (1997), 319--335.

\end{thebibliography}
\end{document}